\documentclass[12pt,a4paper,reqno]{amsart}
\allowdisplaybreaks
\usepackage{amsmath}
\usepackage{amsfonts}
\usepackage{amssymb,amsthm,amsfonts,amsthm,latexsym,enumerate,url,cases}
\usepackage{booktabs,array,longtable}
\numberwithin{equation}{section}
\usepackage{mathrsfs}
\usepackage{hyperref}
\hypersetup{colorlinks=true,citecolor=blue,linkcolor=blue,urlcolor=blue}
\newcommand{\N}{\mathbb N}
\newcommand{\Z}{\mathbb Z}

\newcommand{\ind}{\mathbf{1}}

\theoremstyle{plain}
\newtheorem{theorem}{Theorem}[section]
\newtheorem{lemma}[theorem]{Lemma}

\theoremstyle{definition}

\usepackage{etoolbox}
\makeatletter
\patchcmd{\@settitle}{\uppercasenonmath\@title}{}{}{}
\patchcmd{\@setauthors}{\MakeUppercase}{}{}{}
\patchcmd{\section}{\scshape}{}{}{}
\makeatother

\begin{document}

\title[{On additive complements in natural numbers}]{On additive complements in natural numbers}

\author{Quan-Hui Yang}
\address[Quan-Hui Yang]{Ministry of Education Key Laboratory for NSLSCS \\ School of Mathematical Sciences \\ Nanjing Normal University \\ Nanjing 210023 \\ China}
\email{yangquanhui01@163.com}
\author{Lilu Zhao}
\address[Lilu Zhao]
{School of Mathematical Sciences, University of Science and Technology of China, Hefei, Anhui, 230026, P.R. China}
\email{zhaolilu@ustc.edu.cn}

\keywords{additive complements; counting function; asymptotics, convolution.}
\subjclass[2020]{Primary: 05B10, 11B13, 11B34.}

\begin{abstract}We prove that for additive complements $A$ and $B$, if
$$
 \limsup_{x\to+\infty}\frac{A(x)B(x)}{x}<\frac43,
$$
then
$$
 A(x)B(x)-x\rightarrow+\infty\qquad(x\to+\infty).
$$
This improves the earlier upper bound $3-\sqrt{3}$ due to Fang and Chen.
\end{abstract}

 \maketitle

\section{Introduction}

Let $A,B\subseteq\N:=\{0,1,2,\cdots\}$ be infinite sets. We define the sumset 
$$A+B=\{a+b:\ a\in A,\, b\in B\}.$$ The sets $A$ and $B$ are called \emph{additive complements} if $A+B$ contains
all sufficiently large integers.
 Write
\[
 A(x)=\#\{a\in A:a\leqslant x\},\qquad
 B(x)=\#\{b\in B:b\leqslant x\}.
\]

A problem of Hanani and Erd\H{o}s (see \cite{Erdos1957}), asked whether
$$\limsup_{x\to+\infty}\frac{A(x)B(x)}{x}>1$$ 
must hold for additive complements $A$ and $B$. Danzer \cite{Danzer1964} disproved this in 1964 by constructing complements satisfying
\begin{equation*}
 \frac{A(x)B(x)}{x}\rightarrow 1\quad (x\to \infty).
\end{equation*}
Danzer further conjectured that for additive complements $A$ and $B$, if
\begin{align*}
 \limsup_{x\to+\infty}\frac{A(x)B(x)}{x}\leqslant 1,
\end{align*}
then
$$
 A(x)B(x)-x\rightarrow+\infty\qquad(x\to+\infty).
$$
This conjecture was proved by S\'ark\"ozy and Szemer\'edi \cite{SarkozySzemeredi1994}. The above condition was relaxed to 
\begin{align*}
 \limsup_{x\to+\infty}\frac{A(x)B(x)}{x}<\frac{5}{4}
\end{align*}
by Fang and Chen \cite{FangChen2010}. The constant $\frac{5}{4}$ in the preceding inequality was further improved to $3-\sqrt{3}$ in \cite{FangChen2014}. The purpose of this paper is to prove the following.

\begin{theorem}\label{thm:main}
For additive complements $A$ and $B$, if
\begin{align}\label{condition}
 \limsup_{x\to+\infty}\frac{A(x)B(x)}{x}<\frac43,
\end{align}
then
\begin{align}\label{toinfty}
 A(x)B(x)-x\rightarrow+\infty\qquad(x\to+\infty).
\end{align}
\end{theorem}

Chen and Fang \cite{ChenFang2011} proved that there are additive complements $A$ and $B$ such that
\begin{align*}
 \limsup_{x\to+\infty}\frac{A(x)B(x)}{x}=\frac{3}{2},
\end{align*}
but there exist infinitely many positive integers $x$ such that $A(x)B(x)-x=1$. In fact, a more general result was obtained in \cite{ChenFang2011}. 

For related results in this topic, one may refer to \cite{ChenFang2019,ChenYang2012,Nathanson2011,Ruzsa2001,Ruzsa2017}. 

In Section 2, we prepare some results which are essentially known in \cite{FangChen2014}. In Section 3, we establish an asymptotic formula by the tools of convolution. This is the new ingredient. Finally, we complete the proof of the main theorem in Section 4.
\section{Preparations}
Throughout this paper we assume that the infinite sets $A$ and $B$ are additive complements and \eqref{condition} holds. We also assume that \eqref{toinfty} does not hold. 
Define
\[
 r(n)=\#\{(a,b)\in A\times B:a+b=n\}.
\]
Then $r(n)\geqslant 1$ for all $n\geqslant N_0$. Thus for integer $x\geqslant N_0$, 
\begin{align*}A(x)B(x)\geqslant x-N_0.\end{align*}
By the same arguments as in S\'ark\"ozy and Szemer\'edi \cite[p.238]{SarkozySzemeredi1994} (see also \cite[p.85]{FangChen2014}), there exists $n_0\geqslant N_0$ such that
\begin{equation}\label{eq:uniqueness}
 r(n)=1\qquad(n\geqslant n_0).
\end{equation}
In particular, there is a fixed integer $m_0$ such that
\begin{equation}\label{eq:representation-count}
 \sum_{n=0}^{x}r(n)=x+m_0\qquad(x\in\Z^{+},\ x\geqslant n_0).
\end{equation}
Then by the lower bound above, there exists $L_0$ and infinitely many positive integers $x_j\geqslant n_0$ such that
\begin{align}\label{ABx}A(x_j)B(x_j)-x_j=L_0.\end{align}
In fact, one can take
$$L_0=\liminf_{x\to+\infty}\big(A(x)B(x)-x\big).$$
Let
\[
 T(x)=\#\{(a,b)\in A\times B:a,b\leqslant x,\ a+b>x\}.
\]
Note that $A(x)B(x)=\sum_{n=0}^xr(x)+T(x)$. For integer $x\geqslant x_0$, by \eqref{eq:representation-count}, we have $A(x)B(x)=x+m_0+T(x)$. Thus by \eqref{ABx},
\begin{equation}\label{eq:tail-bound}
 T(x_j)=L_0-m_0.
\end{equation}

We also need the corresponding fact about differences. Put
\[
 d(t)=\#\{(a,b)\in A\times B:b-a=t\}\qquad(t\in\Z).
\]
\begin{lemma}One has \begin{equation*}
 K:=\sum_{t\in \Z}\max\big(d(t)-1,0\big)<+\infty.
\end{equation*}\end{lemma}
\begin{proof}It suffices to prove that $d(t)\leqslant 1$ for $|t|\geqslant n_0$. Suppose otherwise that $d(t)\geqslant 2$ for some $|t|\geqslant n_0$. 
Then $t=b-a=b'-a'$ for two distinct pairs $(a,b)$ and $(a',b')$. Then $a+b'=a'+b$. Recall \eqref{eq:uniqueness}, one has $a+b'=a'+b<n_0$ thus
 $a,a',b,b'<n_0$. This is a contradiction to $|t|\geqslant n_0$. The proof of this lemma is complete.
\end{proof}
In view of \eqref{eq:uniqueness}, we also introduce
\begin{equation*}
 E:=\sum_{n=0}^{+\infty}|r(n)-1|=\sum_{n=0}^{x_0-1}|r(n)-1|.
\end{equation*}

\section{An asymptotic formula}

Recall the notations $n_0$, $K$ and $E$ introduced in Section 2. Note that $n_0,K,E$ are all $O(1)$ terms and $A(x)=o(x)$ as $x\to +\infty$. In this section, we establish the following asymptotic formula. 
\begin{lemma}\label{lemmaasymp}Suppose that  $x\in A$, $x>n_0$ and  $A\cap (x,2x]=\varnothing$. 
One has 
\begin{equation*}
 \left|\frac{1}{A(x)}\sum_{\substack{a\in A \\ a\leqslant x}}a-\frac{x}{2}\right|
 \leqslant \frac{A(x)K+n_0+K+E}{2}.
\end{equation*}
\end{lemma}
\begin{proof}For any functions $f,g$ supported on $\N$, we consider the convolution
\[
 (f*g)(n)=\sum_{u=0}^{n}f(u)g(n-u).
\]
Now let $f=\ind_{A\cap[0,x]}$, $g=\ind_B$ and $h=\ind_{\N}$ be indicator functions. 
Write 
$$\delta(n)=r(n)-h(n).$$
 By
the definition of $E$, 
\[
 E=\sum_{n= 0}^{+\infty}|\delta(n)|.
\]
Define $\widetilde{f}(n)=f(x-n)$, and set
\begin{align*}
 D=h-\widetilde{f}*g.
\end{align*}
For $n\geqslant 0$,
\begin{align*}
 (\widetilde{f}*g)(n)=\sum_{u=0}^n\widetilde{f}(u)g(n-u)=&\,\sum_{u=0}^nf(x-u)g(n-u)\notag
 \\=&\,\#\{(a,b)\in A'\times B:b-a=n-x\}\notag
 \\ \leqslant&\, d(n-x),
\end{align*}
where 
$$A'=A\cap[0,x].$$ 

Write $D_+=\max(D,0)$ and $D_{-}=\max(-D,0)$ pointwise. 
We obtain
from the definition of $K$ and the preceding estimate for $\widetilde{f}*g$,
\begin{equation}\label{eq:D-negative}
\sum_{n= 0}^{+\infty}D_-(n)\leqslant K.
\end{equation}

Since $A\cap(x,2x]=\varnothing$, we have $f*g=h+\delta$ on $[0,2x]$. We deduce that
\begin{align*}
 f*D=f*(h-\widetilde{f}*g)=f*h-f*\widetilde{f}*g=&\,f*h-\widetilde{f}*f*g
\\=&\,(f-\widetilde{f})*h+\widetilde{f}*(h-f*g),
\end{align*}
and thus on $[0,2x]$,
\begin{equation}\label{eq:convolution}
 f*D=
 (f-\widetilde{f})*h-\widetilde{f}*\delta.
\end{equation}
For $x+n_0\leqslant n\leqslant 2x$, both $(f*h)(n)$ and $(\widetilde{f}*h)(n)$ are equal to $A(x)$, and thus $(f-\widetilde{f})*h(n)=0$. Recall that $\delta$ is supported in $[0,n_0-1]$ and $\widetilde{f}$ is supported in $[0,x]$, we also have $(\widetilde{f}*\delta)(n)=0$ for $x+n_0\leqslant n\leqslant 2x$. Then we obtain from \eqref{eq:convolution} that
\begin{equation*}
 (f*D)(n)=0\qquad(x+n_0\leqslant n\leqslant 2x).
\end{equation*}
Since $f(x)=1$, we have $f*D_{+}(n)\geqslant D_{+}(n-x_0)$ and then by the preceding identity
\begin{align*}
 \sum_{n=n_0}^{x}D_+(n)
 \leqslant \sum_{n=x+n_0}^{2x}(f*D_+)(n)
 =\sum_{n=x+n_0}^{2x}(f*D_-)(n).
\end{align*}
We further deduce that
\begin{align*}
\sum_{n=x+n_0}^{2x}(f*D_-)(n)\leqslant \sum_{n=0}^{2x}(f*D_-)(n)= \sum_{u\leqslant 2x}D_{-}(u)\sum_{a\leqslant 2x-u}f(a)
  \leqslant A(x)\sum_{u\leqslant 2x}D_{-}(u).
\end{align*}
and by \eqref{eq:D-negative},
\begin{align*}
\sum_{n=x+n_0}^{2x}(f*D_-)(n) \leqslant A(x)K.
\end{align*}
Similarly, we also have
\begin{align*}\sum_{n=0}^{x}|(f*D)(n)|\leqslant A(x)\sum_{n=0}^{x}D(n) \quad \textrm{ and }\quad
 \sum_{n=0}^{x}|\widetilde{f}*\delta(n)|
  \leqslant A(x)E.
\end{align*}
We conclude from above that
\begin{align}\label{boundD+}
 \sum_{t=n_0}^{x}D_+(t)
  \leqslant A(x)K.
\end{align}

 By the definition of $D$ and $D_{+}$,
\begin{equation*}
 D_+(n)\leqslant 1.
\end{equation*}
Using this pointwise bound on $D_+$ for the remaining $n_0$ indices and noting that $D=D_{+}-D_{-}$,  we deduce from \eqref{eq:D-negative} and \eqref{boundD+} that
\begin{equation*}
 \sum_{n=0}^{x}|D(n)|\leqslant \sum_{n=0}^{x}D_{+}(n)+\sum_{n=0}^{x}D_{-}(n)\leqslant A(x)K+n_0+K.
\end{equation*}

Recall that $A'=A\cap[0,x]$, we have
\begin{align*}
 \sum_{n=0}^{x}\big((f-\widetilde{f})*h\big)(n)
 &=\sum_{u\in A'}(x-u+1)-\sum_{u\in A'}(u+1)\\
 &=A(x)x-2\sum_{u\in A'}u.
\end{align*}
and by \eqref{eq:convolution},
\begin{align*}
 \sum_{n=0}^{x}f*D+\sum_{n=0}^{x}\widetilde{f}*\delta
=A(x)x-2\sum_{u\in A'}u.
\end{align*}
Now the above bounds give
\begin{align*}
 \left|A(x)x-2\sum_{u\in A'}u\right|
 &\leqslant\sum_{n=0}^{x}|(f*D)(n)|
       +\sum_{n=0}^{x}|(\widetilde{f}*\delta)(n)|\\
 &\leqslant A(x)\big(A(x)K+n_0+K\big)+A(x)E.
\end{align*}
Dividing by $2A(x)$ yields the desired estimate. This completes the proof of the lemma.
\end{proof}

\section{Proof of Theorem \ref{thm:main}}

\noindent{\it Proof of Theorem \ref{thm:main}.} We proceed the proof by contradiction, that is we assume that \eqref{toinfty} does not hold. In view of \eqref{condition}, we can choose $0<\kappa<\frac{4}{3}$ such that 
\begin{align}\label{boundbykappa}A(x)B(x)\leqslant \kappa x \ (x\geqslant x_0).\end{align}
Then we have 
\begin{align*}A(x)=o(x)\quad  \textrm{and} \quad B(x)=o(x).\end{align*}

Consider $x_j$ satisfying \eqref{eq:tail-bound} and sufficiently large. We may assume that $x_j>x_0$ for all $j$. In particular both $A(x_j)$ and $B(x_j)$ are sufficiently large. Let 
\[
 u_j=A(x_j)-A(x_j/2),\qquad
 v_j=B(x_j)-B(x_j/2).
\]
For $a\in A$ with $x_j/2<a\leqslant x_j$ and $b\in B$ with $x_j/2<b\leqslant x_j$, one has $a+b>x_j$. So
\[
 u_jv_j\leqslant T(x_j).
\]
We claim that either $A\cap(x_j/2,x_j]=\varnothing$ for infinitely many $j$ or $B\cap(x_j/2,x_j]=\varnothing$ for infinitely many $j$.  Otherwise, both $u_j$ and $v_j$ are positive for infinitely many $j$. Then 
$$u_j,v_j\leqslant L_0-M_0.$$
Then by \eqref{ABx}, we deduce that
\begin{align*}
 A(x_j/2)B(x_j/2)
 &=(A(x_j)-u_j)(B(x_j)-v_j)\\
 &=A(x_j)B(x_j)-u_jB(x_j)-v_jA(x_j)+u_jv_j\\
 &=(1+o(1))x_j.
\end{align*}
This is a contradiction to \eqref{boundbykappa}. Therefore, the above claim holds. Interchanging $A$ and $B$ if
necessary, we may assume that for infinitely many $j$,
\begin{align}\label{nothing}
 A\cap(x_j/2,x_j]=\varnothing.
\end{align}
From now on, we focus on those $j$ satisfying \eqref{nothing}.
Let $a_j=\max(A\cap[0,x_j])$. Then $a_j\to+\infty$ as $j\to+\infty$. Moreover, $a_j\leqslant x_j/2$, and
\begin{equation*}
 A\cap(a_j,2a_j]=\varnothing.
\end{equation*}
 Applying Lemma \ref{lemmaasymp} with $x=a_j$, we have
\begin{equation*}
 \frac{1}{A(x)}\sum_{\substack{u\in A\\u\leqslant x}}u=\frac x2+O\big(\frac{x}{B(x)}\big)=\frac{x}{2}+o(x).
\end{equation*}

By \eqref{eq:representation-count},  
\begin{equation*}
 2x+m_0=\sum_{n=0}^{2x}r(n)=\sum_{\substack{u\in A\\u\leqslant x}}B(2x-u).
\end{equation*}
For every $u$ in this sum, $x\leqslant 2x-u\leqslant 2x$, and therefore
$A(2x-u)=A(x)$. By \eqref{boundbykappa},
\[
 B(2x-u)\leqslant \frac{\kappa(2x-u)}{A(x)}.
\]
Combining this with the preceding two identities gives
\begin{align*}
 2x+m_0
 \leqslant\frac{\kappa}{A(x)}
        \sum_{\substack{u\in A\\u\leqslant x}}(2x-u)
 =2\kappa x-\frac{\kappa}{A(x)}
        \sum_{\substack{u\in A\\u\leqslant x}}u=\kappa\left(\frac32x+o(x)\right).
\end{align*}
Dividing by $x$ and letting $x=a_j\to\infty$, we obtain
\[
 2\leqslant\frac32\kappa,
\]
which is a contradiction to $\kappa<4/3$. We complete the proof of Theorem \ref{thm:main}.

\section*{Acknowledgments}
This work is support by the National Key Research and Development Program of China (Grant No. 2021YFA1000701), National Natural Science Foundation of China  (Grant No. 12371005 and 12471088).


\begin{thebibliography}{99}


\bibitem{ChenFang2011}
Y.-G.~Chen and J.-H.~Fang,
\emph{On additive complements, II},
Proc. Amer. Math. Soc. \textbf{139} (2011), 881--883.

\bibitem{ChenFang2019}
Y.-G.~Chen and J.-H.~Fang,
\emph{Additive complements with Narkiewicz's condition},
Combinatorica \textbf{39} (2019), 813--823.

\bibitem{ChenYang2012}
Y.-G.~Chen and Q.-H.~Yang,
\emph{On a problem of Nathanson related to minimal additive complements},
SIAM J. Discrete Math. \textbf{26} (2012), 1532--1536.

\bibitem{Danzer1964}
L.~Danzer,
\emph{\"Uber eine Frage von G.~Hanani aus der additiven Zahlentheorie},
J. Reine Angew. Math. \textbf{214/215} (1964), 392--394.

\bibitem{Erdos1957}
P.~Erd\H{o}s,
\emph{Some unsolved problems},
Michigan Math. J. \textbf{4} (1957), 291--300.

\bibitem{Nathanson2011}
M.~B.~Nathanson,
\emph{Problems in additive number theory, IV: Nets in groups and shortest length $g$-adic representations},
Int. J. Number Theory \textbf{7} (2011), 1999--2017.


\bibitem{FangChen2010}
J.-H.~Fang and Y.-G.~Chen,
\emph{On additive complements},
Proc. Amer. Math. Soc. \textbf{138} (2010), 1923--1927.

\bibitem{FangChen2014}
J.-H.~Fang and Y.-G.~Chen,
\emph{On additive complements. III},
J. Number Theory \textbf{141} (2014), 83--91.


\bibitem{Ruzsa2001}
I.~Z.~Ruzsa,
\emph{Additive completion of lacunary sequences},
Combinatorica \textbf{21} (2001), 279--291.


\bibitem{Ruzsa2017}
I.~Z.~Ruzsa,
\emph{Exact additive complements},
Q. J. Math. \textbf{68} (2017), 227--235.

\bibitem{SarkozySzemeredi1994}
A.~S\'ark\"ozy and E.~Szemer\'edi,
\emph{On a problem in additive number theory},
Acta Math. Hungar. \textbf{64} (1994), 237--245.

\end{thebibliography}
\end{document}